\documentclass[12pt]{amsart}

\newtheorem{theorem}{Theorem}[section]
\newtheorem{lemma}[theorem]{Lemma}

\newtheorem{corollary}[theorem]{Corollary}

\usepackage[PostScript=dvips,noPS]{diagrams}

\newtheorem{exampleit}[theorem]{Example}
\newenvironment{examplerm}{\begin{exampleit}\rm}{\end{exampleit}}

\newcommand{\CC}{\mathbb C}
\newcommand{\RR}{\mathbb R}
\newcommand{\ZZ}{\mathbb Z}

\newarrow{Equalto}{=}{=}{=}{=}{=}

\newcommand{\pt}{\mathord{*}}

\newcommand{\wwedge}{\vee}

\title{Bundles over Connected Sums}

\author{Lisa C. Jeffrey}
\address{Department of Mathematics, University of Toronto,
Canada}
\email{jeffrey@math.toronto.edu}

\author{Paul Selick}
\address{Department of Mathematics, University of Toronto,
Canada}
\email{selick@math.toronto.edu}

\begin{document}

\begin{abstract}
A principal bundle over the connected sum of two manifolds need not be 
diffeomorphic or even homotopy equivalent to a non-trivial connected sum
of manifolds.
We show however that the homology of the total space of a bundle formed as
a pullback of a bundle over one of the summands is the same as if it had
that bundle as a summand.
See Theorem~\ref{decomposition}.
An application appears in~\cite{HJSX}.

Examples are given, including one where the total space of the
pullback is not homotopy equivalent to a connected sum with that as a summand
and some in which it is.

Finally, we describe the homology of the total space of a principal
$U(1)$~bundle over a 6-manifold of the type described by Wall's theorem.
It is a connected sum of an even number of copies of $S^3 \times S^4$ with 
a $7$-manifold whose homology is $\ZZ/k$ in degree $4$ (and $\ZZ$ in 
degrees $0$ and $7$, and zero in all other degrees).

\end{abstract}

\maketitle

\section{Introduction}

Let $A$ be a connected sum $A\cong B\#C$ of $n$-manifolds.
See for example Hatcher \cite{Hatcher} for the definition of connected sum.
Let $F\to L\to C$ be a bundle over~$C$ where $F$ is a manifold.

Using the definition we get a map $A\to C$.
Let $F\to M\to A$ be the pullback of the bundle $F\to L\to C$ to~$A$.

Letting $B'$ denote the complement of a chart in~$B$ and setting
$X':=(B'\times F)/(\pt\times F)$
we prove the following.
There is a cofibration $X'\to M\to L$ for which the corresponding long
exact homology/cohomology sequences split to give
$H_*(M)\cong H_*(X')\oplus H_*(L)$ and
$H^*(M)\cong H^*(X')\oplus H^*(L)$.
(See Theorems \ref{cofibration} and~\ref{decomposition}.)

These results suggest the possibility that $M$ is the connected sum of~$L$
and some manifold~$X$ whose $(n-1)$ skeleton is homotopy equivalent to~$X'$
but we give an example to show that this is not necessarily the case.
(See Example~\ref{counterex}).
As we shall see, if $M\simeq X\#L$ then the cofibration sequence
$X'\to M\to L$ would have to split to give $M\simeq X'\wwedge L$, but this
fails in Example~\ref{counterex}.

In the final section, we consider bundles over some $6$-manifolds
including the case where $A$ is a symplectic manifold and the $M$ is the
total space of its associated prequantum line bundle.
We find that in that case $M\simeq \#^{2r}(S^3\times S^4)\#L$
where $L$ is a $7$-manifold whose nonzero cohomology groups are $\ZZ$ in
degrees $0, 7$ and $\ZZ/k$ in degree $4$, where $r$ and~$k$ are determined
by the cohomology of~$A$.
(See Theorem~\ref{decompm}.)

For topological spaces $X$ and~$Y$, let $X\cong Y$ denote ``$X$ is homeomorphic
to~$Y$'' and let $X\simeq Y$ denote ``$X$ is homotopy equivalent to~$Y$''.

\section{Connected Sums}

Let $D^n$ denote the closed disk $D^n:=\{x\in \RR^n\mid\|x\|=1\}$.

\begin{lemma}
For any points $a$, $b$ in the interior of~$D^n$ there exists a
self-diffeomorphism $f:D^n\to D^n$ such that $f(a)=b$ and
$f|_{\partial D^n}$ is the identity. 
\end{lemma}

\begin{proof}
Set $f(a)=b$.
For $x\ne a$, let $X_x$ be the point at which the production of the
line segment joining $a$ to~$x$ meets $\partial D^n$.
Then $x=ta+(1-t)X_x$ for some~$t$.
Set $f(x)=tb+(1-t)X_x$.
\end{proof}

More generally, we have
\begin{lemma}
Let $U_p$, $V_q$ be subcharts of~$D^n$.
Then there exists a self-diffeomorphism $f:D^n\to D^n$ such that
the restriction of $f$ to $U_p$ is the standard diffeomorphism on
open balls and such that $f|_{\partial D^n}$ is the identity.
\end{lemma}

For a point~$p$ in an $n$-manifold~$X$, define a subchart around~$p$
to be an open neighbourhood $U_p$ of~$p$ which is diffeomorphic to an
open ball in $\RR^n$ within some chart of~$X$.

For a connected $n$-manifold $X$, let $X'=X\setminus D^n$ denote the complement
of a subchart of~$X$.

\begin{lemma}
Up to diffeomorphism, $X'$ is independent of the choice of the subchart
removed.
\end{lemma}

\begin{proof}
Let $U_p$, $U_q$ be subcharts of~$X$.
In the special case where there exists a chart~$W$ containing
both~$\bar{U_p}$ and $\bar{U_q}$ this follows from the earlier lemma.
Then for arbitrary $U_p$, $U_q$, find a finite (by compactness) chain of
charts connecting $U_p$ to~$V_q$, using connectivity.

\end{proof}

After removal of the subchart there is a deformation
retraction $X'\simeq X^{(n-1)}$ to the $(n-1)$-skeleton of~$X$.
Let $f_X:S^{n-1}\to X'$ denote the attaching map of the top cell of~$X$.

Suppose that $X$, $Y$ are simply connected oriented $n$-manifolds.

In a connected sum, $X\#Y=X'\cup_{S^{n-1}\times I}Y'$
(where the orientation on one of the inclusions $S^{n-1}\times\{0\}\rInto X'$
or $S^{n-1}\times\{1\}$ is reversed so that $X\#Y$ inherits an orientation),
there is a canonical projection $X\#Y\to X$.
Similarly we have $X\#Y\to Y$.
The canonical projections $X\#Y\to X$ and $X\#Y\to Y$
preserve the orientation class.
That is, they induce isomorphisms on~$H_n(~)$. 

Collapsing the centre of the tube $S^{n-1}\times I$ within $X\#Y$ gives
a map $X\#Y\to X'\wwedge Y'$.
If we form $(X\#Y)'$ by choosing the subchart to be removed to be
within the centre of the tube then collapsing to produce~$X'\wwedge Y'$
has collapsed a contractible subset of~$(X\# Y)'$ giving a homotopy
equivalence $(X\#Y)'\rTo^{\simeq} X'\wwedge Y'$.

By writing $S^n=S^n\#S^n$ and considering naturality of the pinch we see that
the homotopy class of the attaching map of the top cell in $X\#Y$ is given
by $f_{X\#Y}=f_X+f_Y$ within
$\pi_n(X)\oplus\pi_n(Y)\subset \pi_n(X'\wwedge Y')$.

Choosing the subchart to be removed from~$X\#Y$ to be within~$Y'$
gives a (non-canonical) inclusion $X'\rInto (X\#Y)'$
with $(X\#Y)'/X'\cong Y'$.
The composite $X'\rInto (X\#Y)'\to X$ with the canonical projection
is an injective map from a compact Hausdorff space, so it is a homeomorphism
to its image.
Composing with the inverse of this homeomorphism is a left splitting of
the inclusion $X'\rInto (X\#Y)'$.
Similarly there is a left splitting of the inclusion $Y'\rInto (X\#Y)'$.

\begin{lemma}
Let $M$ be a closed $n$-manifold and let $A\subset M'$ be a closed $n$-dim
subset of $M$ with
$\partial \bar{A} \cong S^{n-1}$.
Then $M\cong N\#X$ for some manifolds $N$ and~$X$ with $N'=A$.
Furthermore the canonical projection
$M'\to N'=A$ is a left splitting of the inclusion $A\rInto M'$.
\end{lemma}

\begin{proof}
Set $\hat{X}:=M\setminus A$.
Then $\hat{X}$ is a manifold-with-boundary with
$\partial\hat{X}=\partial \bar{A}$.
Let $T\cong S^{n-1}\times I$ be a tubular neighbourhood of
$\partial \hat{X}$ in $\hat{X}$ and set $X':=\hat{X}\setminus T$.
Then
$$M=\bar{A}\cup_{S^{n-1}\times \{0\}} T
\cup_{S^{n-1}\times \{1\}}\overline{X'}$$
so $M=N\# X$ where $N=A\cup_{(T\times \{0\})}D^n$ and
$X=X'\cup_{(T\times \{1\})}D^n$.
By construction $A\rInto M'\to N'=A$ is the identity on~$A$.
\end{proof}

\section{The cofibration sequence associated to a bundle over a connected sum}

For definitions and properties of principal cofibrations used in this
section see pp. 56--61 of~\cite{selickbook}.

Let $B$, $C$ be closed $n$-manifolds and let $A:=B\#C$.
Suppose $$F \to L \to C$$ is a (locally trivial) fibre bundle
whose fibre~$F$ is a manifold.
Then $L$ is a manifold of dimension $n + \dim (F)$, which we will denote
by~$m$.

Let $F\to M\to A$ be the pullback of the bundle under the canonical
projection $A\to C$.
The total space $M$ is a manifold of dimension $m$.

Let $\hat{L}$ be the total space of the restriction of
the bundle to $C':=C\setminus{\rm chart}$.

By definition,
$$ A = B' \cup_{S^{n-1} \times I} C'$$
where by construction, the restriction of the bundle to $B'$ is
trivial.

Taking inverse images under the bundle projection $M\rOnto A$
gives
$$ M = (B' \times F) \cup_{(S^{n-1} \times I  \times F)} \hat{L}. $$

In other words, we have
\begin{diagram}
S^{n-1} \times I \times F & \rInto & (B' \times F)& \rTo & (B' \times F)/(S^{n-1} \times I \times F) \cr
\dTo           &&                 \dTo            && \dEqualto \cr
\hat{L} & \rInto & M &\rTo & M/\hat{L}\cr
\end{diagram}
where the left square is a pushout.

The space
$$
\begin{array}{ll}
M/\hat{L}&=(B'\times F)/(S^{n-1}\times I\times F)\cr
&=\bigl(B'/(S^{n-1}\times I) \times F\bigr)/(* \times F)\cr
&=(B\times F)/(*\times F).\cr
\end{array}$$
has the same homology as~$ B \wwedge (B \wedge  F). $
In fact, if $F$ is a suspension then 
$(B \times F)/(* \times F) \simeq
B \wwedge (B \wedge  F). $
(Selick,~\cite{selickbook}~Prop~7.7.8)

Set $X':=(B'\times F)/(*\times F)$.

\begin{theorem}\label{cofibration}
There is a cofibration diagram
\begin{diagram}
X'&\rTo& M'&\rTo& L'\cr
\dEqualto&&\dTo&&\dTo\cr
X'&\rTo& M&\rTo& L\cr
\end{diagram}
(i.e. the rows are cofibrations and the right square is a pushout.)
\end{theorem}

\begin{proof}
We had
$$M/\hat{L}=(B'\times F)/(S^{n-1}\times I\times F)$$
Also, since $L=\hat{L}\cup_{S^{n-1}\times I\times F}F$ we have
$$L/\hat{L}=(S^n\times F)/(*\times F)$$
(which can be regarded as the special case of the preceding
with $B=S^n$).
Thus we have a diagram
\begin{diagram}
&&X'&\rEqualto& X'\cr
&&\dTo &&\dTo\cr
\hat{L}&\rTo& M&\rTo&M/\hat{L}=(B\times F)/(*\times F)\cr
\dEqualto&&\dTo &&\dTo\cr
\hat{L}&\rTo& L&\rTo&L/\hat{L}=(S^n\times F)/(*\times F)\cr
\end{diagram}
in which the bottom right square is a pushout, the rows and right columns
are cofibrations and which yields the cofibration
$X'\to M\to L$.
Deleting a chart from $L$ and deleting its preimage from~$M$
gives the first row of the theorem.
\end{proof}

 From the long exact homology sequence of the cofibration we get
\begin{corollary}
The lift $M\to L$ of the canonical projection preserves the
orientation class.
That is, it induces isomorphisms on $H_{m}(~)$, where $m$ was defined earlier.
\end{corollary}
This Corollary can be proved in other ways such as naturality of
the Serre spectral sequence.

Let $f:X\to Y$ be a differentiable map between compact oriented
$m$-manifolds.
Let $D_X:H^k(X)\cong H_{n-k}(X)$ and $D_Y:H^k(Y)\cong H_{n-k}(Y)$
be the Poincar\'e Duality isomorphisms.
Suppose $f$ has degree~$\lambda $ (multiplies by
$\lambda$ on~$H_n(~)$).
Then $f_*\circ D_X\circ f^*=\lambda D_Y$.
In particular, if $f$ preserves the orientation class (that is,
has degree~$1$)  then $f^*$ is injective and $f_*$ is surjective.
Applying this to $M\to L$ shows
\begin{theorem}(Decomposition Theorem)\label{decomposition}
{\ } 

In the long exact homology sequence of the cofibration,
the connecting map $\partial:H_q(L)\to H_{q-1}(X')$ is zero. 
Likewise, in the
long exact cohomology sequence,  the map
$\delta:H^{q-1}(X')\to H^q(L)$ is zero.
Thus for $0<q<m$ we have
 $H_q(M)\cong H_q(X')\oplus H_q(L)$
and  $H^q(M)\cong H^q(X')\oplus H^q(L)$.
\end{theorem}

This suggests that perhaps there is a manifold~$X$ such that
$M\simeq X\#L$ where $X$ is
homotopy equivalent to the one-point compactification
of~$X'$,
but this is not necessarily true.

\begin{examplerm}\label{counterex}
Consider $A=\CC P^2$ and write $A=B\#C$ where $B=\CC P^2$ and $C=S^4$
Consider the trivial bundle $S^7\times S^4 \to S^4$.
Then $M=S^7\times \CC P^2$; $B'=S^2$; $C'=*$; $A'=B'\vee C'=S^2$ while
$$M'=(F\times A)'=(F\times A')\cup_{F'\times A'} (F'\times A) $$
$$=(S^7\times S^2)\cup_{*\times S^2} (*\times \CC P^2)
=\CC P^2\wwedge S^7\wwedge S^9
$$
and $L=S^4\times S^7$ so $L'=S^4\wwedge S^7$.
Our cofibration is
$$(S^2\times S^7)/(*\times S^7)\to M'\to S^4\wwedge S^7$$
which becomes
$S^2\wwedge S^9\to \CC P^2\wwedge S^7\wwedge S^9\to S^4\wwedge S^7$.
This does not split so 
in this example
$M$ does not become homotopy equivalent to~ $X\#L$
for any~$X$.
\end{examplerm}

\section{Bundles over $6$-manifolds}

Let $A$ be a $6$-manifold such that $H^*(A)$ is simply connected and
torsion-free.
Suppose $H^2(A)=\ZZ$.

Let $x\in H^2(A)$ be a generator and let $V\in H^6(A)$ be the volume form.
Then $x^3=kV$ for some integer~$k$.

By Wall~\cite{Wall}, we can write $A=B\#C$ where $B=(S^3\times S^3)^{\#r}$ for
some~$r$ and $C$ is a simply connected torsion-free $6$-manifold
with $H^3(C)=0$ and~$H^2(C)=\ZZ$.

Although $M$ is a $S^1$ bundle over $A$, it does not immediately 
follow from Wall's result that $M$ also admits a decomposition as a
connected sum. We shall see that this is in fact true. 
This is the content of our Theorem \ref{decompm} below.

Associated to $x$ there are  complex line bundles over
$A$ and $C$ classified by~$x$.
Let $M$ and $L$ denote the sphere bundles of these line bundles.
Then there are $S^1$-bundles $S^1\to M\to A$ and $S^1\to L\to C$.
Note that the long exact homotopy sequence tells us that
$\pi_1(M)=\ZZ$ and $\pi_q(M)=\pi_q(A)$ for $q\ne1$.

As in Ho-Jeffrey-Selick-Xia~\cite{HJSX} we calculate that the cohomology of
the $7$-manifold $L$ is given by
$H^q(L)=\begin{cases}
\ZZ& q=0,7;\cr
\ZZ/k& q=4.\cr 
0&otherwise.\cr
\end{cases}$

\begin{theorem}\label{decompm}
We have
$$M \simeq	\#^{2r}(S^3\times S^4)\#L,$$
where the homology of the space $L$ is specified above.
\end{theorem}

\begin{proof}
In the notation of the preceding section applied to $S^1\to M\to A$
we have $B'=\wwedge_{2r}S^3$, $L'=P^4(k)$ and
$$X':=(B'\times S^1)/(\pt\times S^1)\simeq B'\wwedge (B'\wedge S^1)
\wwedge_{2r} (S^3 \wwedge \Sigma^3 S^1)
$$
where $P^n(k)$ denotes the Moore space $S^{n-1}\cup_k e^n$.
Thus our cofibration sequence becomes
$$\wwedge_{2r} (S^3 \wwedge \Sigma^3 S^1)\to M'\to P^4(k)$$
or equivalently
$$\wwedge_{2r} (S^3 \wwedge S^4)\to M'\to P^4(k).$$

The composition of the bundle map $M'\to A'$ with the canonical
projection $A'\to B'$ provides a splitting of  
 the restriction of
$$\wwedge_{2r} (S^3 \wwedge S^4)\to M'$$ to~$\wwedge_{2r} S^3$.

For degree reasons, the cofibration
$$\wwedge_{2r} (S^3 \wwedge S^4)\to M'\to P^4(k)$$
is principal, induced from some attaching map
$P^3(k)\to \wwedge_{2r} (S^3 \wwedge S^4)$
whose image (for degree reasons) lands in $\wwedge_{2r} S^3$.
Since the restriction of
$\wwedge_{2r} (S^3 \wwedge S^4)\to M'$ to~$\wwedge_{2r} S^3$
splits, this implies that this attaching map is trivial.
Thus the cofibration splits to give
$$M'\simeq \wwedge_{2r} (S^3 \wwedge S^4)\wwedge_{2r} P^4(k).$$

To obtain $M$ from $M'$ we attach the top cell giving
$$H^q(M)=H^q(M')\oplus H^q(S^7)
=\begin{cases}
\ZZ&q=0,7\cr
\ZZ^{2r}&q=3\cr
\ZZ^{2r}\oplus \ZZ/k&q=4\cr
0&{\rm otherwise}.\cr
\end{cases}
$$
Letting $\tilde{V}$ denote the generator of $H^7(M)$,
using Poincar\'e duality we can pair the generators
$\langle u_1, u_2,\ldots u_r\rangle$ of $\ZZ$ in degrees $3$
with the generators
$\langle v_1, v_2,\ldots v_r\rangle$ of $\ZZ$ in degrees $4$
so that $u_i v_j=\delta_{ij}\tilde{V}.$
If we reduce to $\ZZ/k$ coefficients, there is also a nonzero
cup product $ab$ where $a$, $b$ are generators of $H^3(M;\ZZ/k)$
and $H^4(M;\ZZ/k)$ respectively.

Examining the cohomology of~$M$, we see that
$$H^*(M)=H^*\bigl(\#^{2r}(S^3\times S^4)\#L\bigr)$$
where
$H^q(L)
=\begin{cases}
\ZZ&q=0,7\cr
\ZZ/k&q=4\cr
0&{\rm otherwise}.\cr
\end{cases}$
\end{proof}
The attaching maps $f_M$ and $f_{\#^{2r}(S^3\times S^4)\#L}$
are both
$$[\iota^3_1,\iota^4_1]+[\iota^3_2,\iota^4_2]
+\ldots+[\iota^3_r,\iota^4_r]+f_L$$
where the Whitehead product $[\iota^3,\iota^4]$ is the attaching
map~
$$f_{S^3\times S^4},$$ and so
$$M\simeq \#^{2r}(S^3\times S^4)\#L. $$
\hfill $\square$

\end{document}